\documentclass[12pt]{article}
\usepackage{mathrsfs}
\usepackage{graphicx,amsmath,amsfonts,amssymb,amscd,amsthm}
\usepackage[notref,notcite]{showkeys}
\newtheoremstyle{kai}
{3pt} {3pt} {} {} {\bfseries} {.} {.5em} {}
\def\EquationsBySection{\def\theequation
{\thesection.\arabic{equation}}%
\@addtoreset{equation}{section}}

\newcommand\old[1]{}
\newcommand{\pend}{\hfill \thicklines \framebox(6.6,6.6)[l]{}}

\renewenvironment{proof}{\noindent {\it  Proof.} \rm}{\pend}
\newtheorem{theorem}{Theorem}[section]
\newtheorem{lemma}{Lemma}[section]
\newtheorem{corollary}{Corollary}[section]

\newtheorem{remark}{Remark}[section]
\newtheorem{definition}{Definition}[section]

\EquationsBySection \makeatother
\usepackage{indentfirst}
\newcommand{\disp}{\displaystyle}

\begin{document}
\pagestyle{plain}
\title
{\bf Competitive Lotka-Volterra Population Dynamics with Jumps}

\author{Jianhai Bao$^{1,3}$, Xuerong Mao$^2$,
 Geroge Yin$^3$, Chenggui Yuan$^4$
\\
\\
$^1$School of Mathematics, Central South University,\\
 Changsha, Hunan
410075, P.R.China\\
jianhaibao@yahoo.com.cn\\[1ex]
 $^2$Department of Statistics and Modelling Science,\\
University of Strathclyde, Glasgow G1 1XH, UK
\\
xuerong@stams.strath.ac.uk\\[1ex]
  $^3$Department of Mathematics,\\
Wayne State University, Detroit, Michigan 48202.
\\
gyin@math.wayne.edu\\[1ex]
$^4$Department of Mathematics, Swansea University,\\
 Swansea SA2 8PP, UK\\
 C.Yuan@swansea.ac.uk
 }

\date{}
\maketitle
\begin{abstract}
This paper considers competitive Lotka-Volterra population dynamics
with jumps. The contributions of this paper are as follows. (a) We
show stochastic differential equation (SDE) with jumps associated
with the model has a unique global positive solution; (b) We discuss
the uniform boundedness of $p$th moment with $p>0$ and reveal the
sample Lyapunov exponents; (c) Using a variation-of-constants
formula for a class of SDEs with jumps, we provide explicit solution
for $1$-dimensional competitive Lotka-Volterra population dynamics
with jumps, and investigate  the sample Lyapunov exponent for each
component and  the extinction of our $n$-dimensional model.

\medskip

\noindent {\bf Keywords.}   Lotka-Volterra Model, Jumps, Stochastic
Boundedness, Lyapunov Exponent, Variation-of-Constants Formula,
Stability in Distribution, Extinction.

\medskip

\noindent{\bf Mathematics Subject Classification (2010).} \ 93D05,
60J60, 60J05.
\end{abstract}
\newpage

\section{Introduction}
The differential equation
\begin{equation*}
\begin{cases}
\dfrac{dX(t)}{dt}&  =X(t)[a(t)-b(t)X(t)],\ \ \ t\geq0,\\
X(0)& = x,
\end{cases}
\end{equation*}
 has been used to model
the population growth of a single species whose members usually live
in proximity, share the same basic requirements, and compete for
resources, food, habitat, or territory, and is known as the
competitive Lotka-Volterra model or logistic equation. The
competitive Lotka-Volterra model for $n$ interacting species is
described by the $n$-dimensional differential equation
\begin{equation}\label{eq08}
\frac{dX_i(t)}{dt}=X_i(t)\left[a_i(t)-\sum\limits_{j=1}^nb_{ij}(t)X_j(t)\right],i=1,2,\cdots,
n,
\end{equation}
where $X_i(t)$ represents the population size of species $i$ at time
$t$, $a_i(t)$ is the rate of growth at time $t$, $b_{ij}(t)$
represents the effect of interspecific (if $i\neq j$) or
intraspecific (if $i=j$) interaction at time $t$, $a_i(t)/b_{ij}(t)$
is the carrying capacity of the $i$th species in absence of other
species at time $t$. Eq. \eqref{eq08}  takes the matrix form
\begin{equation}\label{eq09}
\frac{dX(t)}{dt}=\mbox{diag}(X_1(t),\cdots,X_n(t))\left[a(t)-B(t)X(t)\right],
\end{equation}
where
\begin{equation*}
X=(X_1,\cdots,X_n)^T,a=(a_1,\cdots,a_n)^T, B=(b_{ij})_{n\times n}.
\end{equation*}
There is an extensive literature concerned with the dynamics of Eq.
\eqref{eq09} and we here only mention Gopalsamy \cite{g92}, Kuang
\cite{k93}, Li et al. \cite{ltj99}, Takeuchi and Adachi \cite{ta80,
ta78}, Xiao and Li \cite{xl00}. In particular, the books by
Gopalsamy \cite{g92}, and Kuang \cite{k93} are good references in
this area.

On the other hand, the deterministic models assume that parameters
in the systems are all deterministic irrespective environmental
fluctuations, which, from the points of biological view, has some
limitations in mathematical modeling of ecological systems. While,
population dynamics in the real world is
 affected inevitably by environmental noise, see, e.g., Gard
 \cite{g84,g86}. Therefore, competitive Lotka-Volterra models in
 random environments are becoming more and more popular. In general, there are two  ways  considered in the literature to
 model the influence of environmental fluctuations in population
 dynamics. One is to consider the random perturbations of interspecific
 or intraspecific interactions by white noise. Recently, Mao et al. \cite{mmr02} investigate
 stochastic $n$-dimensional Lotka-Volterra system
\begin{equation}\label{eq0}
dX(t)=\mbox{diag}(X_1(t),\cdots,X_n(t))\left[(a+BX(t))dt+\sigma
X(t)dW(t)\right],
\end{equation}
where $W$ is a one-dimensional standard Brownian motion, and reveal
that the environmental noise can suppress a potential population
explosion (see, e.g., \cite{myz05, my06} among others in this
connection). Another is to consider the stochastic perturbation of
growth rate $a(t)$ by the white noise with
\begin{equation*}
a(t)\rightarrow a(t)+\sigma(t)\dot{W}(t),
\end{equation*}
where $\dot{W}(t)$ is a white noise, namely, $W(t)$ is a Brownian
motion defined on a complete probability space $(\Omega, \mathcal
{F},\mathbb{P})$ with a filtration $\{\mathcal {F}\}_{t\geq0}$
satisfying the usual conditions (i.e., it is right continuous and
increasing while $\mathcal {F}_0$ contains all $\mathbb{P}$-null
sets). As a result, Eq. \eqref{eq09} becomes a competitive
Lotka-Volterra model in
 random environments
\begin{equation}\label{eq41}
dX(t)=\mbox{diag}(X_1(t),\cdots,X_n(t))\left[(a(t)-B(t)X(t))dt+\sigma(t)
dW(t)\right].
\end{equation}
There is also extensive literature concerning all kinds of
properties of model \eqref{eq41}, see, e.g.,  Hu and Wang
\cite{hw10}, Jiang and Shi \cite{js05}, Liu and Wang \cite{lw11},
Zhu and Yin \cite{zy09a, zy09b}, and the references therein.

Furthermore, the population may suffer sudden environmental shocks,
e.g., earthquakes, hurricanes, epidemics, etc. However, stochastic
Lotka-Volterra model \eqref{eq41} cannot explain such phenomena. To explain these phenomena, introducing a jump process into
underlying population dynamics provides a  feasible and more
realistic model.
In this paper, we
develop Lotka-Volterra model with jumps\begin{equation}\label{eq42}
\begin{split}
dX(t)&=\mbox{diag}(X_1(t^-),\cdots,X_n(t^-))\Big[(a(t)-B(t)X(t))dt\\
&\quad+\sigma(t)
dW(t)+\int_{\mathbb{Y}}\gamma(t,u)\tilde{N}(dt,du)\Big].
\end{split}
\end{equation}
Here $X, a, B$ are defined as in Eq. \eqref{eq09},
\begin{equation*}
\sigma=(\sigma_1,\cdots,\sigma_n)^T,
\gamma=(\gamma_1,\cdots,\gamma_n)^T,
\end{equation*}
 $W$ is a real-valued standard Brownian motion, $N$ is a Poisson counting
 measure with characteristic measure $\lambda$ on a measurable subset
$\mathbb{Y}$ of $[0,\infty)$ with $\lambda(\mathbb{Y})<\infty$,
 $\tilde{N}(dt,du):=N(dt,du)-\lambda(du)dt$. Throughout the paper, we
 assume that $W$ and $N$ are independent.

 As we know,
for example, bees colonies in a field \cite{r83}. In particular,
they compete for food strongly with the colonies located near to
them. Similar phenomena abound in the nature, see, e.g., \cite{r79}.
Hence it is reasonable to assume that the self-regulating
competitions within the same species are strictly positive, e.g.,
\cite{zy09a, zy09b}. Therefore we also assume

\begin{description}
\item{${\bf(A)}$} For any $t\geq0$ and $i,j=1,2,\cdots,n$ with $i\neq
j$, $a_i(t)>0, b_{ii}(t)>0, b_{ij}(t)\geq0, \sigma_i(t)$ and $\gamma_i(t, u)$ are bounded
functions, $\hat{b}_{ii}:=\inf_{t\in \mathbb{R}_+}b_{ii}(t)>0$ and
$\gamma_i(t,u)>-1,u\in\mathbb{Y}$.
\end{description}

In reference to the existing results in the literature, our
contributions are as follows:
\begin{itemize}
\item We use jump diffusion to model the evolutions of
population dynamics;
\item We demonstrate that if the population dynamics with jumps is
self-regulating or competitive, then the population will not explode
in a finite time almost surely;
\item We discuss the uniform boundedness of $p$-th moment for any $p>0$ and reveal the sample Lyapunov exponents;
\item We obtain the explicit expression of $1$-dimensional
competitive Lotka-Volterra model with jumps, the uniqueness of invariant measure,
and further reveal precisely the sample Lyapunov exponents for each
component and investigate its extinction.
\end{itemize}

\section{Global Positive Solutions}
As the $i$th state $X_i(t)$ of Eq. \eqref{eq42} denotes the size of
the $i$th species in the system, it should be nonnegative. Moreover,
in order to guarantee SDEs to have a unique global (i.e., no
explosion in a finite time) solution for any given initial data, the
coefficients of the equation are generally required to satisfy the
linear growth and local Lipschitz conditions, e.g., \cite{my06}.
However, the drift coefficient of Eq. \eqref{eq42} does not satisfy
the linear growth condition, though it is locally Lipschitz
continuous, so the solution of Eq. \eqref{eq42} may explode in a
finite time. It is therefore requisite to provide some conditions
under which the solution of Eq. \eqref{eq42} is not only positive
but will also not explode to infinite in any finite time.

Throughout this paper, $K$ denotes a generic constant whose values
may vary for its different appearances. For a bounded function $\nu$
defined on $\mathbb{R}_+$, set
\begin{equation*}
\hat{\nu}:=\inf_{t\in\mathbb{R}_+}\nu(t) \mbox{ and }
\check{\nu}:=\sup_{t\in\mathbb{R}_+}\nu(t).
\end{equation*}

For convenience of reference, we recall some fundamental
inequalities stated as a lemma.
\begin{lemma}
{\rm \begin{equation}\label{eq100} x^r\leq1+r(x-1),\ \ \ \ x\geq0,\
\ \ 1\geq r\geq0,
\end{equation}
\begin{equation}\label{eq101}
n^{(1-\frac{p}{2})\wedge0}|x|^{p}\leq\sum\limits_{i=1}^nx_i^p\leq
n^{(1-\frac{p}{2})\vee0}|x|^{p}, \forall p>0, x\in\mathbb{R}^n_+,
\end{equation}
where $\mathbb{R}^n_+:=\{x\in \mathbb{R}^n:x_i>0, 1\leq i\leq n\}$,
and
\begin{equation}\label{eq102}
\ln x\leq x-1,\ \ \ \ \  \ \ \ \ \ \ \ \ \ \ \ \  x>0.
\end{equation}

}

\end{lemma}

\begin{theorem}\label{positive solution}
{\rm Under assumption ${\bf(A)}$, for any initial condition $
X(0)=x_0\in\mathbb{R}^n_+$, Eq. \eqref{eq42} has a unique global
solution $X(t)\in\mathbb{R}^n_+$ for any $t\geq0$ almost surely.}
\end{theorem}

\begin{proof}
Since the drift coefficient does not fulfil the linear growth
 condition, the general theorems of existence and uniqueness cannot
 be implemented to this equation. However, it is locally Lipschitz continuous, therefore for any given
 initial condition $X(0)\in
\mathbb{R}^n_+$ there is a unique local solution
 $X(t)$ for $t\in[0,\tau_e)$, where $\tau_e$ is the explosion time. By Eq. \eqref{eq42} the $i$th component
$X_i(t)$ of $X(t)$  admits the form for $i=1,\cdots,n$
\begin{equation*}
dX_i(t)=X_i(t^-)\Big[\Big(a_i(t)-\sum\limits_{j=1}^nb_{ij}(t)X_j(t)\Big)dt+\sigma_i(t)dW(t)
+\int_{\mathbb{Y}}\gamma_i(t,u)\tilde{N}(dt,du)\Big].
\end{equation*}
 Noting that for any $t\in[0,\tau_e)$
\begin{equation*}
\begin{split}
X_i(t)=X_i(0)\exp\Big\{&\int_0^t\Big(a_i(s)-\sum\limits_{j=1}^nb_{ij}(s)X_j(s)-\frac{1}{2}
\sigma_i^2(s)\\
&+\int_\mathbb{Y}(\ln(1+\gamma_i(s,u))-\gamma_i(s,u))\lambda(du)\Big)ds\\
&+\int_0^t\sigma_i(s)dW(s)+\int_0^t\int_\mathbb{Y}\ln(1+\gamma_i(s,u))\tilde{N}(ds,du)\Big\},
\end{split}
\end{equation*}
together with $X_i(0)>0$, we can conclude $X_i(t)\geq0$ for any
$t\in[0,\tau_e)$. Now consider the following two auxiliary SDEs with
jumps
\begin{equation}\label{eq103}
\begin{array}{ll}
dY_i(t)&\!\!\!\disp=
Y_i(t^-)\Big[\Big(a_i(t)-b_{ii}(t)Y_i(t)\Big)dt+\sigma_i(t)dW(t)
+\int_{\mathbb{Y}}\gamma_i(t,u)\tilde{N}(dt,du)\Big],\\
Y_i(0)&\!\!\!= X_i(0),
\end{array}
\end{equation}
and
\begin{equation}\label{eq112}
\begin{array}{ll}
dZ_i(t)&\!\!\!\disp= Z_i(t^-)\Big[\Big(a_i(t)-\sum\limits_{i\neq
j}b_{ij}(t)Y_j(t)-b_{ii}(t)Z_i(t)\Big)dt+\sigma_i(t)dW(t)
+\int_{\mathbb{Y}}\gamma_i(t,u)\tilde{N}(dt,du)\Big],\\
Z_i(0)&\!\!\!= X_i(0).
\end{array}
\end{equation}
Due to $1+\gamma_i(t,u)>0$ by ${\bf(A)}$, it follows that for any
$x_2\geq x_1$
\begin{equation*}
(1+\gamma_i(t,u))x_2\geq(1+\gamma_i(t,u))x_1.
\end{equation*}
Then by the comparison theorem \cite[Theorem 3.1]{pz06} we can
conclude  that
\begin{equation}\label{eq111}
Z_i(t)\leq X_i(t)\leq Y_i(t),t\in[0,\tau_e).
\end{equation}
By Lemma \ref{explicit solution} below, for $Y_i(0)(=X_i(0))>0,$ we know that $Y_i(t)$ will not be expolded in any finite time. Moreover, similar to that of Lemma \ref{explicit solution} below for $Z_i(0)(=X_i(0))>0,$ we can show
\begin{equation*}
\mathbb{P}(Z_i(t)>0 \mbox{ on } t\in[0,\tau_e))=1.
\end{equation*}
Hence $\tau_e=\infty$ and $X_i(t)>0$ almost surely for any $t\in[0,\infty)$.
The proof is therefore complete.
\end{proof}

\section{Boundedness, Tightness, and Lyapunov-type Exponent}

In the previous section, we see that Eq. \eqref{eq42} has a unique
global solution $X(t)\in\mathbb{R}^n_+$ for any $t\geq0$ almost
surely. In this part we shall show for any $p>0$ the solution $X(t)$
of Eq. \eqref{eq42} admits uniformly finite $p$-th moment, and
discuss the long-term behaviors.

\begin{theorem}\label{finite moment}
{\rm Let assumption ${\bf(A)}$  hold.
\begin{description}
\item[(1)] For any $p\in [0,1, ]$ there is a constant $K$ such that
\begin{equation}\label{eq9501}
\sup_{t\in\mathbb{R}_+}\mathbb{E}|X(t)|^p\leq K.
\end{equation}
\item[(2)] Assume further that there exists a
constant $\bar{K}(p)>0$ such that for some $p>1,t\geq0,i=1,\cdots,n$
\begin{equation}\label{eq90}
\int_{\mathbb{Y}}|\gamma_i(t,u)|^p\lambda(du)\leq \bar{K}(p).
\end{equation}
Then  there exists a constant $K(p)>0$ such that
\begin{equation}\label{eq95}
\sup_{t\in\mathbb{R}_+}\mathbb{E}|X(t)|^p\leq K(p).
\end{equation}
\end{description} }
\end{theorem}

\begin{proof}  We shall prove (\ref{eq95}) firstly.
Define a Lyapunov function for
$p>1$
\begin{equation}\label{eq43}
V(x):=\sum\limits_{i=1}^nx_i^p, x\in\mathbb{R}^n_+.
\end{equation}
 Applying the It\^o formula, we obtain
\begin{equation*}
\mathbb{E}(e^tV(X(t)))=V(x_0)+\mathbb{E}\int_0^te^s[V(X(s))+\mathcal
{L}V(X(s),s)]ds,
\end{equation*}
where, for $x\in\mathbb{R}^n_+$ and $t\geq0$,
\begin{equation}\label{eq50}
\begin{split}
\mathcal {L}V(x,t)&:=p\sum\limits_{i=1}^n
\left[a_i(t)-\sum\limits_{j=1}^nb_{ij}(t)x_j-\frac{(1-p)\sigma_i^2(t)}{2}\right]x_i^{p}
\\
&\quad+\sum\limits_{i=1}^n\int_{\mathbb{Y}}\left[(1+\gamma_i(t,u))^{p}-1
-p\gamma_i(t,u)\right]\lambda(du)x_i^{p}.
\end{split}
\end{equation}
 By
assumption ${\bf(A)}$ and \eqref{eq90}, we can deduce that there
exists constant $K>0$ such that
\begin{equation*}
\begin{split}
V(x)+\mathcal {L}V(x,t)&\leq \sum\limits_{i=1}^n\left[-pb_{ii}(t)
x_i^{p+1}+\left(1+pa_i(t)+\frac{p(p-1)\sigma_i^2(t)}{2}\right)x_i^{p}\right]\\
&\quad+\sum\limits_{i=1}^n\int_{\mathbb{Y}}\left[(1+\gamma_i(t,u))^{p}-1
-p\gamma_i(t,u)\right]\lambda(du)x_i^{p}\\
&\leq K.
\end{split}
\end{equation*}
Hence
\begin{equation*}
\mathbb{E}(e^tV(X(t)))\leq V(x_0)+\int_0^tKe^sds=V(x_0)+K(e^t-1),
\end{equation*}
which yields the desired assertion (\ref{eq95}) by the inequality \eqref{eq101}.

For any $p\in[0,1]$, according to the inequality \eqref{eq100},
\begin{equation*}
\int_{\mathbb{Y}}\left[(1+\gamma_i(t,u))^{p}-1
-p\gamma_i(t,u)\right]\lambda(du)\leq0.
\end{equation*}
Consequently
\begin{equation*}
\begin{split}
V(x)+\mathcal {L}V(x,t)\leq \sum\limits_{i=1}^n\left[-pb_{ii}(t)
x_i^{p+1}+\left(1+pa_i(t)\right)x_i^{p}\right],
\end{split}
\end{equation*}
which has upper bound by ${\bf(A)}$. Then \eqref{eq9501} holds with
$p\in[0,1]$ under  ${\bf(A)}$.
\end{proof}

\begin{corollary}\label{exin}
{\rm Under assumption ${\bf(A)}$, there exists an invariant
probability measure for the solution $X(t)$ of Eq. \eqref{eq42}.}
\end{corollary}

\begin{proof}
Let $\mathbb{P}(t,x,A)$ be the transition probability measure of
$X(t,x)$, starting from $x$ at time $0$. Denote
\begin{equation*}
\mu_T(A):=\frac{1}{T}\int_0^T\mathbb{P}(t,x,A)dt
\end{equation*}
and $B_r:=\{x\in\mathbb{R}_+^n:|x|\leq r\}$ for $r\geq0$. In the
light of  Chebyshev's inequality and Theorem \ref{finite moment}
with $p\in(0,1)$,
\begin{equation*}
\mu_T(B_r^c)=\frac{1}{T}\int_0^T\mathbb{P}
(t,x,B_r^c)dt\leq\frac{1}{r^pT}\int_0^T
\mathbb{E}|X(t,x)|^pdt\leq\frac{K}{r^p},
\end{equation*}
and we have, for any $\epsilon>0$, $\mu_T(B_r)>1-\epsilon$ whenever
$r$ is large enough. Hence $\{\mu_T, T>0\}$ is tight. By
Krylov-Bogoliubov's theorem, e.g.,  \cite[Corollary3.1.2, p22]{pz97},
the conclusion follows immediately.
\end{proof}

\begin{definition}
{\rm The solution $X(t)$ of Eq. \eqref{eq42} is called
stochastically bounded, if for any $\epsilon\in(0,1)$, there is a
constant $H:=H(\epsilon)$ such that for any  $x_0\in\mathbb{R}^n_+$
\begin{equation*}
\limsup\limits_{t\rightarrow\infty}\mathbb{P}\{|X(t)|\leq
H\}\geq1-\epsilon.
\end{equation*}
}
\end{definition}

As an application of  Theorem \ref{finite moment}, together with the
Chebyshev inequality, we can also establish the following corollary.

\begin{corollary}\label{ultimate boundedness}
{\rm Under  assumption ${\bf(A)}$, the solution $X(t)$ of Eq.
\eqref{eq42} is stochastically bounded.}
\end{corollary}

For later applications, let us cite a strong law of large numbers
for local martingales, e.g., Lipster \cite{l80}, as the following
lemma.
\begin{lemma}\label{large numbers}
{\rm Let $M(t),t\geq0$, be a local martingale vanishing at time $0$
and define
\begin{equation*}
\rho_M(t):=\int_0^t\frac{d\langle M\rangle(s)}{(1+s)^2}, t\geq0,
\end{equation*}
where $\langle M\rangle(t):=\langle M,M\rangle(t)$ is Meyer's angle
bracket process. Then
\begin{equation*}
\lim\limits_{t\rightarrow\infty}\frac{M(t)}{t}=0 \mbox{ a.s.
provided that } \lim\limits_{t\rightarrow\infty}\rho_M(t)<\infty
\mbox{ a.s. }
\end{equation*}
}
\end{lemma}

\begin{remark}
{\rm Let
\begin{equation*}
\Psi^2_{\mbox{loc}}:=\left\{\Psi(t,z) \mbox{ predictable }
\Big|\int_0^t\int_{\mathbb{Y}}|\Psi(s,z)|^2\lambda(du)ds<\infty\right\}
\end{equation*}
and for $\Psi\in\Psi^2_{\mbox{loc}}$
\begin{equation*}
M(t):=\int_0^t\int_{\mathbb{Y}}\Psi(s,z)\tilde{N}(ds,du).
\end{equation*}
Then, by, e.g., Kunita \cite[Proposition 2.4]{k10}
\begin{equation*}
\langle
M\rangle(t)=\int_0^t\int_{\mathbb{Y}}|\Psi(s,z)|^2\lambda(du)ds
\mbox{ and } [M](t)=\int_0^t\int_{\mathbb{Y}}|\Psi(s,z)|^2N(ds,du),
\end{equation*}
where $ [M](t):=[M,M](t)$, square bracket process (or quadratic
variation process) of $M(t)$. }
\end{remark}

\begin{theorem}
{\rm Let assumption ${\bf(A)}$ hold. Assume further that for some
constant $\delta>-1$ and any $t\geq0$
\begin{equation}\label{eq78}
\gamma_i(t,u)\geq\delta,u\in \mathbb{Y}, i=1,\cdots,n,
\end{equation}
and there exists constant $K>0$ such that
\begin{equation}\label{eq99}
\int_0^t\int_{\mathbb{Y}}|\gamma(s,u)|^2\lambda(du)ds\leq Kt.
\end{equation}
 Then  the solution
$X(t),t\geq0$, of Eq. \eqref{eq42} has the property
\begin{equation}\label{eq16}
\limsup\limits_{t\rightarrow\infty}\frac{1}{t}\left[\ln(|X(t)|)+\frac{\min\limits_{1\leq
i\leq n}\hat{b}_{ii}}{\sqrt{n}}\int_0^t|X(s)|ds\right]\leq
\max_{1\leq i\leq n}\check{a}_i, \ \ \ \mbox{a.s.}
\end{equation}
}
\end{theorem}

\begin{proof}
 For any $x\in \mathbb{R}^n_+$, let $
V(x)=\sum\limits_{i=1}^nx_i$, by It\^o's formula
\begin{equation*}
\begin{split}
\ln(V(X(t)))&\leq\ln(V(x_0))+\int_0^t\Big(X^T(s)(a(s)-B(s)X(s)/V(X(s))\\
&\ \ \ \ \ \ \ \ \ \ \ \ \ \ \ \ \ \ \ \ \ \ \ \ \ -(X^T(s)\sigma(s)
)^2/(2V^2(X(s)))\Big)ds\\
&\quad+\int_0^tX^T(s)\sigma(s)/V(X(s))dW(s)
+\int_0^t\int_{\mathbb{Y}}\ln(1+H(X(s^-),s,u))\tilde{N}(ds,du),
\end{split}
\end{equation*}
where
\begin{equation*}
H(x,t,u)= \left(\sum\limits_{i=1}^n\gamma_i(t,u)x_i\right)\Big/V(x).
\end{equation*}
Here we used the fact that $1+H>0$ and the inequality \eqref{eq102}.
Note from the inequality \eqref{eq101} and assumption ${\bf(A)}$
that
\begin{equation*}
\begin{split}
&X^T(s)(a(s)-B(s)X(s))/V(X(s))-(X^T(s)\sigma(s) )^2/(2V^2(X(s))\\
&\leq\frac{\sum\limits_{i=1}^na_i(s)X_i(s)}{\sum\limits_{i=1}^nX_i(s)}-\frac{\sum\limits_{i=1}^nX_i(s)\sum\limits_{j=1}^nb_{ij}(s)X_j(s)}{\sum\limits_{i=1}^nX_i(s)}\\
&\leq\max_{1\leq i\leq n}\check{a}_i-\frac{\min_{1\leq i\leq
n}\hat{b}_{ii}}{\sqrt{n}}|X(s)|.
\end{split}
\end{equation*}
Let
\begin{equation*}
M(t):=\int_0^tX^T(s)\sigma(s)/V(X(s))dW(s) \mbox{ and }
\tilde{M}(t):=\int_0^t\int_{\mathbb{Y}}\ln(1+H(X(s^-),s,u))\tilde{N}(ds,du).
\end{equation*}
Compute by the boundedness of $\sigma$ that
\begin{equation*}
\langle
M\rangle(t)=\int_0^t(X^T(s)\sigma(s))^2/V^2(X(s))ds\leq\int_0^t|\sigma(s)|^2ds\leq
Kt.
\end{equation*}
On the other hand, by assumption \eqref{eq78} and the definition of
$H$, for $x\in\mathbb{R}_+^n$ we obtain
\begin{equation*}
H(x,t,u)\geq\delta
\end{equation*}
and, in addition to \eqref{eq102}, for $-1<\delta\leq0$
\begin{equation*}
\begin{split}
|\ln(1+H(x,t,u))|&\leq|\ln(1+H(x,y,u))I_{\{\delta\leq
H(x,t,u)\leq0\}}|
+|\ln(1+H(x,y,u))I_{\{0\leq H(x,t,u)\}}|\\
&\leq-\ln(1+\delta)+|H(x,t,u)|.
\end{split}
\end{equation*}
This, together with \eqref{eq99}, gives that
\begin{equation*}
\begin{split}
\langle
\tilde{M}\rangle(t)&=\int_0^t\int_{\mathbb{Y}}(\ln(1+H(X(s),s,u)))^2\lambda(du)ds\\
&\leq2(-\ln(1+\delta))^2\lambda(\mathbb{Y})t+2\int_0^t\int_{\mathbb{Y}}H^2(X(s),s,u))\lambda(du)ds\\
&\leq2(-\ln(1+\delta))^2\lambda(\mathbb{Y})t+2\int_0^t\int_{\mathbb{Y}}|\gamma(t,u)|^2\lambda(du)ds\\
&\leq(2(-\ln(1+\delta))^2\lambda(\mathbb{Y})+K)t.
\end{split}
\end{equation*}
Then the strong law of large numbers, Lemma \ref{large numbers},
yields
\begin{equation*}
\frac{1}{t}M(t)\rightarrow0 \mbox{ a.s. and }
\frac{1}{t}\tilde{M}(t)\rightarrow0 \mbox{ as } t\rightarrow\infty,
\end{equation*}
and the conclusion follows.
\end{proof}

\section{Variation-of-Constants Formula and the Sample Lyapunov Exponents}
 In this part we
further discuss the long-term behaviors of model \eqref{eq42}.
 To begin, we
 obtain the following variation-of-constant formula for $1$-dimensional diffusion with jumps, which is
interesting in its own right.

\subsection{Variation-of-Constants Formula}

\begin{lemma}\label{variation-of-constant formula}
{\rm Let $F,G,f,g:\mathbb{R}_+\rightarrow\mathbb{R}$ and
$H,h:\mathbb{R}_+\times\mathbb{Y}\rightarrow\mathbb{R}$ be
Borel-measurable and bounded functions with property $H>-1$, and
$Y(t)$
satisfy 
\begin{equation}\label{eq3}
\begin{split}
dY(t)&=[F(t)Y(t)+f(t)]dt+[G(t)Y(t)+g(t)]dW(t)\\
&\quad +\int_{\mathbb{Y}}[Y(t^-)H(t,u)+h(t,u)]\tilde{N}(dt,du),\\
Y(0)&=Y_0.
\end{split}
\end{equation}
Then the solution can be explicitly expressed as:
\begin{equation*}
\begin{split}
Y(t)&=\Phi(t)\Big(Y_0+\int_0^t\Phi^{-1}(s)\Big[\Big(f(s)-G(s)g(s)-
\int_{\mathbb{Y}}\frac{H(s,u)h(s,u)}{1+H(s,u)}\lambda(du)\Big)ds\\
&\quad
+g(s)dW(s)+\int_{\mathbb{Y}}\frac{h(s,u)}{1+H(s,u)}\tilde{N}(ds,du)\Big]\Big),
\end{split}
\end{equation*}
where
\begin{equation*}
\begin{split}
\Phi(t)&:=\exp\Big[\int_0^t\Big(F(s)-\frac{1}{2}G^2(s)+
\int_{\mathbb{Y}}[\ln(1+H(s,u))-H(s,u)]\lambda(du)\Big)ds\\
&\quad+\int_0^tG(s)dW(s)
+\int_0^t\int_{\mathbb{Y}}\ln(1+H(s,u))\tilde{N}(ds,du)\Big]
\end{split}
\end{equation*}
is the fundamental solution of corresponding homogeneous linear
equation
\begin{equation}\label{eq8}
dZ(t)=F(t)Z(t)dt+G(t)Z(t)dW(t)+Z(t^-)\int_{\mathbb{Y}}H(t,u)\tilde{N}(dt,du).
\end{equation}
}
\end{lemma}

\begin{proof}
 Noting that
\begin{equation*}
\begin{split}
\Phi(t)&=\exp\Big[\int_0^t\Big(F(s)-\frac{1}{2}G^2(s)+
\int_{\mathbb{Y}}[\ln(1+H(s,u))-H(s,u)]\lambda(du)\Big)ds\\
&\quad +\int_0^tG(s)dW(s)
+\int_0^t\int_{\mathbb{Y}}\ln(1+H(s,u))\tilde{N}(ds,du)\Big]
\end{split}
\end{equation*}
is the fundamental solution to Eq. \eqref{eq8}, we then have
\begin{equation}\label{eq4}
d\Phi(t)=F(t)\Phi(t)dt
+G(t)\Phi(t)dW(t)+\Phi(t^-)\int_{\mathbb{Y}}H(t,u)\tilde{N}( dt,du).
\end{equation}
By \cite[Theorem 1.19, p10]{os05}, Eq. \eqref{eq3} has a unique
solution $Y(t), t\geq0$.
We assume that
\begin{equation*}
Y(t)=\Phi(t)\left(Y(0)+\int_0^t\Phi^{- 1}(s)\left[\bar{f}(s)ds
+\bar{g}(s)dW(s)+\int_{\mathbb{Y}}\bar{h}(s,u)\tilde{N}(d
s,du)\right]\right),
\end{equation*}
where $\bar{f}$, $\bar{g}$, and $\bar{h}$ are functions to be
determined. Let
\begin{equation*}
\bar{Y}(t)=Y(0)+\int_0^t\Phi^{-
1}(s)\left[\bar{f}(s)ds+\bar{g}(s)dW(s)
+\int_{\mathbb{Y}}\bar{h}(s,u)\tilde{N}(d s,du)\right],
\end{equation*}
which means
\begin{equation}\label{eq5}
d\bar{Y}(t)=\Phi^{- 1}(t)\left[\bar{f}(t)dt
+\bar{g}(t)dW(t)+\int_{\mathbb{Y}}\bar{h}(t,u)\tilde{N}(d
t,du)\right].
\end{equation}
Observing that $\Phi$ and $\bar{Y}$ are real-valued L\'{e}vy type
stochastic integrals, by It\^o's product formula, e.g.,
\cite[Theorem 4.4.13, p231]{a09}, we can deduce that
\begin{equation}\label{eq6}
dY(t)=\Phi(t^-)d\bar{Y}(t)+\bar{Y}(t^-)d\Phi(t)+d[\Phi,\bar{Y}](t),
\end{equation}
where $[\Phi,\bar{Y}]$ is the cross quadratic variation of processes
$\Phi$ and $\bar{Y}$, and by $(4.14)$ in \cite[p230]{a09}
\begin{equation}\label{eq7}
d[\Phi,\bar{Y}](t)=G(t)\bar{g}(t)dt+\int_{\mathbb{Y}}H(t,u)\bar{h}(t,u)N(dt,du).
\end{equation}
Putting \eqref{eq4}, \eqref{eq5}, and \eqref{eq7} into \eqref{eq6},
we deduce that
\begin{equation*}
\begin{split}
dY(t)&=\left[\bar{f}(t)dt+\bar{g}(t)dW(t)+\int_{\mathbb{Y}}\bar{h}(t,u)\tilde{N}(dt,d
u)\right]\\
&\quad+F(t)Y(t)dt+G(t)Y(t)dW(t)+Y(t^-)\int_{\mathbb{Y}}H(t,u)\tilde{N}(dt,du)\\
&\quad+G(t)\bar{g}(t)dt+\int_{\mathbb{Y}}H(t,u)\bar{h}(t,u)N(dt,du)\\
&=\left[\bar{f}(t)+F(t)Y(t)+G(t)\bar{g}(t)+\int_{\mathbb{Y}}H(t,u)\bar{h}
(t,u)\lambda(du)\right]dt+[\bar{g}(t)+G(t)Y(t)]dW(t)\\
&\quad
+\int_{\mathbb{Y}}\left[\bar{h}(t,u)+Y(t^-)H(t,u)+H(t,u)\bar{h}(t,u)\right]\tilde
{N}(dt,du).
\end{split}
\end{equation*}
Setting
\begin{equation*}
\bar{f}(t)+G(t)\bar{g}(t)
+\int_{\mathbb{Y}}H(t,u)\bar{h}(t,u)\lambda(du)= f(t)
\end{equation*}
and
\begin{equation*}
\bar{g}(t)=g(t) \mbox{ and } \bar{h}(t,u)+H(t,u)\bar{h}(t,u)=h(t,u),
\end{equation*}
hence we derive that
\begin{equation*}
\bar{f}(t)=f(t)-G(t)g(t)-
\int_{\mathbb{Y}}\frac{H(t,u)h(t,u)}{1+H(t,u)}\lambda(du),
\bar{g}(t)=g(t) \mbox { and } \bar{h}(t,u)=\frac{h(t,u)}{1+H(t,u)}
\end{equation*}
and the required expression follows.
\end{proof}

\subsection{One Dimensional Competitive Model}

In what follows,  we shall  study some properties of processes $Y_i(t)$ defined by \eqref{eq103}, which is actually one dimensional competitive model.

\begin{lemma}\label{explicit solution}
{\rm Under assumption ${\bf(A)}$, Eq. \eqref{eq103} admits a unique
positive solution $Y_i(t),t\geq0$, which admits the explicit formula
\begin{equation}\label{eq9}
Y_i(t)=\frac{\Phi_i(t)}
{\frac{1}{X_i(0)}+\int_0^t\Phi_i(s)b_{ii}(s)ds},
\end{equation}
where
\begin{equation*}
\begin{split}
\Phi_i(t)&:=\exp\Big(\int_0^t\Big[a_i(s)-\frac{1}{2}\sigma_i^2(s)+
\int_{\mathbb{Y}}(\ln(1+\gamma_i(s,u))-\gamma_i(s,u))\lambda(du)\Big]ds\\
&\quad+\int_0^t\sigma_i(s)dW(s)
+\int_0^t\int_{\mathbb{Y}}\ln(1+\gamma_i(s,u))\tilde{N}(ds,du)\Big).
\end{split}
\end{equation*}
}
\end{lemma}
\begin{proof}
It is easy to see that $\Phi_i(t)$ is integrable in any finite interval, hence $Y_i(t)$ will never reach $0$. Letting $\bar{Y}_i(t):=\frac{1}{Y_i(t)}$ and applying the It\^o
formula we have
\begin{equation*}
\begin{split}
d\bar{Y}_i(t)&=-\frac{1}{Y_i^2(t)}Y_i(t)[(a_i(t)-
b_{ii}(t)Y_i(t))dt+\sigma_i(t)dW(t)]+\frac{1}{2}
\frac{2}{Y_i^3(t)}\sigma_i^2(t)Y_i^2(t)dt\\
&\quad+\int_{\mathbb{Y}}\left[\frac{1}{(1+\gamma_i(t,u))Y_i(t)}-
\frac{1}{Y_i(t)}+\frac{1}{Y_i^2(t)}Y_i(t)\gamma_i(t,u)\right]\lambda(du)dt\\
&\quad+ \int_{\mathbb{Y}}\left[\frac{1}{(1+\gamma_i(t,u))Y_i(t^-)}-
\frac{1}{Y_i(t^-)}\right]\tilde{N}(dt,du),
\end{split}
\end{equation*}
that is,
\begin{equation}\label{eq2}
\begin{split}
d\bar{Y}(t)&=\bar{Y}(t^-)\Big[\Big(\sigma_i^2(t)-a_i(t)+
\int_{\mathbb{Y}}\Big(\frac{1}{1+\gamma_i(t,u)}-1+
\gamma_i(t,u)\Big)\lambda(du)\Big)dt-\sigma_i(t)dW(t)\\
&\quad+\int_{\mathbb{Y}}\Big(\frac{1}{1+\gamma_i(t,u)}-
1\Big)\tilde{N}(dt,du)\Big]+b_{ii}(t)dt.
\end{split}
\end{equation}
By Lemma \ref{variation-of-constant formula}, Eq. \eqref{eq2} has an
explicit
solution and the conclusion \eqref{eq9} follows.
\end{proof}

\begin{definition}
{\rm The solution of Eq. \eqref{eq103} is said to be stochastically
permanent if for any $\epsilon\in(0,1)$ there exit positive
constants $H_1:=H_1(\epsilon)$ and $H_2:=H_2(\epsilon)$ such that
\begin{equation*}
\liminf\limits_{t\rightarrow\infty}\mathbb{P}\{Y_i(t)\leq
H_1\}\geq1-\epsilon \mbox{ and
}\liminf\limits_{t\rightarrow\infty}\mathbb{P}\{Y_i(t)\geq
H_2\}\geq1-\epsilon.
\end{equation*}
}
\end{definition}

\begin{theorem}\label{permanent}
{\rm Let assumption ${\bf(A)}$ hold. Assume further that there
exists constant $c_1>0$ such that, for any $t\geq0$ and
$i=1,\cdots,n$,
\begin{equation}\label{eq03}
a_i(t)-\sigma^2_i(t)-
\int_{\mathbb{Y}}\frac{\gamma_i^2(t,u)}{1+\gamma_i(t,u)}\lambda(du)\geq
c_1,
\end{equation}
then the solution $Y_i(t),t\geq0$ of  Eq. \eqref{eq103} is
stochastically permanent.}
\end{theorem}

\begin{proof} The first part of the proof follows by the Chebyshev
inequality and Corollary \ref{ultimate boundedness}. Observe that
\eqref{eq9} can be rewritten in the form
\begin{equation}\label{eq02}
\begin{split}
\frac{1}{Y_i(t)}&=\frac{1}{X_i(0)}\exp\Big(\int_0^t-\Big[a_i(s)-\frac{1}{2}\sigma^2_i(s)+
\int_{\mathbb{Y}}(\ln(1+\gamma_i(s,u))-\gamma_i(s,u))\lambda(du)\Big]ds\\
&\quad-\int_0^t\sigma_i(s)dW(s)
-\int_0^t\int_{\mathbb{Y}}\ln(1+\gamma_i(s,u))\tilde{N}(ds,du)\Big)\\
&\quad+\int_0^tb_{ii}(s)\exp\Big(\int_s^t-\Big[a(r)-\frac{1}{2}\sigma^2_i(r)+
\int_{\mathbb{Y}}(\ln(1+\gamma_i(r,u))-\gamma_i(r,u))\lambda(du)\Big]dr\\
&\quad-\int_s^t\sigma_i(r)dW(r)
-\int_s^t\int_{\mathbb{Y}}\ln(1+\gamma_i(r,u))\tilde{N}(dr,du)\Big)ds.
\end{split}
\end{equation}
By, e.g., \cite[Corollary 5.2.2, p253]{a09}, we notice that
\begin{equation*}
\begin{split}
\exp\Big(&-\frac{1}{2}\int_0^t\sigma_i^2(s)ds-
\int_0^t\int_{\mathbb{Y}}\Big(\frac{1}{1+\gamma_i(s,u)}-
1+\ln(1+\gamma_i(s,u))\Big)\lambda(du)ds\\
&\quad-\int_0^t\sigma_i(s)
dW(s)-\int_0^t\int_{\mathbb{Y}}\ln(1+\gamma_i(s,u))\tilde{N}(ds,du)\Big)
\end{split}
\end{equation*}
is a local martingale. Hence letting
$\bar{M}_i(t):=\frac{1}{Y_i(t)}$ and taking expectations on both
sides of \eqref{eq02} leads to
\begin{equation*}
\begin{split}
\mathbb{E}\bar{M}_i(t)&=\frac{1}{X_i(0)}\exp\Big(-\int_0^t\Big[a_i(s)-\sigma^2_i(s)-
\int_{\mathbb{Y}}\frac{\gamma_i^2(s,u)}{1+\gamma_i(s,u)}\lambda(du)\Big]ds\\
&\quad+\int_0^tb_{ii}(s)\exp\Big(-\int_s^t\Big[a_i(r)-\sigma^2_i(r)-
\int_{\mathbb{Y}}\frac{\gamma_i^2(r,u)}{1+\gamma_i(r,u)}\lambda(du)\Big]drds,
\end{split}
\end{equation*}
which, combining \eqref{eq03}, yields
\begin{equation}\label{eq05}
\mathbb{E}\bar{M}_i(t)\leq\frac{1}{X_i(0)}e^{-c_1t}+\int_0^tb_{ii}(s)e^{-c_2(t-s)}ds\leq\frac{\check{b}}{c_1}+\left(\frac{1}{X_i(0)}
-\frac{\check{b}}{c_1}\right)e^{-c_1t}.
\end{equation}
 Hence there exists a constant
$K>0$ such that
\begin{equation}\label{eq04}
\mathbb{E}\bar{M}_i(t)\leq K.
\end{equation}
Furthermore, for any $\epsilon>0$ and constant $H_2(\epsilon)>0$,
thanks to the Chebyshev inequality and \eqref{eq04}
\begin{equation*}
\mathbb{P}\{Y_i(t)\geq
H_2\}=\mathbb{P}\left\{\bar{M}_i(t)\leq1/H_2\right\}=1-\mathbb{P}\left\{\bar{M}_i(t)>1/H_2\right\}\geq1-H_2\mathbb{E}\bar{M}_i(t)\geq1-\epsilon
\end{equation*}
whenever $H_2=\epsilon/K$, as required.

\end{proof}

\begin{theorem}\label{asymptotic}
{\rm Let the conditions of Theorem \ref{permanent} hold. Then Eq.
\eqref{eq103} has the property
\begin{equation}\label{eq06}
\lim\limits_{t\rightarrow\infty}\mathbb{E}|Y_i(t,x)-Y_i(t,y)|^{\frac{1}{2}}=0
\mbox{ uniformly in } (x,y)\in \mathbb{K}\times \mathbb{K},
\end{equation}
where $\mathbb{K}$ is any compact subset of $(0,\infty)$.}
\end{theorem}

\begin{proof}
By the H\"older inequality
\begin{equation*}
\begin{split}
\mathbb{E}|Y_i(t,x)-
Y_i(t,y)|^{\frac{1}{2}}&=\mathbb{E}\left(Y_i(t,x)Y_i(t,y)\left|\frac{1}{Y_i(t,y)}-
\frac{1}{Y_i(t,x)}\right|\right)^{\frac{1}{2}}\\
&\leq(\mathbb{E}(Y_i(t,x)Y_i(t,y)))^{\frac{1}{2}}\left(\mathbb{E}\left|\frac{1}{Y_i(t,
y)}-\frac{1}{Y_i(t,x)}\right|\right)^{\frac{1}{2}}.
\end{split}
\end{equation*}
To show the desired assertion it is sufficient to estimate the two
terms on the right-hand side of the last step.
By virtue of the It\^o formula,
\begin{equation*}
\begin{split}
d(Y_i(t,x)Y_i(t,y))&=Y_i(t^-,x)dY_i(t,y)+Y_i(t^-,y)dY_i(t,x)+d[Y_i(t,x),Y_i(t,y)]\\
&=Y_i(t^-,x)Y_i(t^-,y)\left[(a_i(t)-b_{ii}(t)Y_i(t,y))dt+\sigma_i(t)
dW(t)+\int_{\mathbb{Y}}\gamma_i(t,u) \tilde{N}(dt,du)\right]\\
&\quad+Y_i(t^-,x)Y_i(t^-,y)\left[(a_i(t)-b_{ii}(t)Y_i(t,x))dt+\sigma_i(t)
dW(t)+\int_{\mathbb{Y}}\gamma_i(t,u) \tilde{N}(dt,du)\right]\\
&\quad+\sigma^2_i(t)Y_i(t,x)Y_i(t,y)dt+\int_{\mathbb{Y}}\gamma^2_i(t,u)Y_i(t^-,x)Y_i(t^-,y)N(dt,du)\\
&=(2a_i(t)+\sigma^2_i(t))Y_i(t,x)Y_i(t,y)dt-b_{ii}(t)Y_i(t,x)Y_i(t,y)(Y_i(t,x)+Y_i(t,y))dt\\
&\quad +2\sigma_i(t)
Y_i(t,x)Y_i(t,y)dW(t)+2\int_{\mathbb{Y}}\gamma_i(t,u)
Y_i(t^-,x)Y_i(t^-,y)\tilde{N}(dt,du)\\
&\quad+\int_{\mathbb{Y}}\gamma^2_i(t,u)Y_i(t^-,x)Y_i(t^-,y)N(dt,du).
\end{split}
\end{equation*}
Thus, in view of Jensen's inequality and the
familiar inequality $a+b\geq2\sqrt{ab}$ for any $a,b\geq0$, we
deduce that
\begin{equation*}
\begin{split}
\mathbb{E}(Y_i(t,x)Y_i(t,y))&\leq xy+\int_0^t\delta_i(s)\mathbb{E}(Y_i(s,x)Y_i(s,y))ds\\
&\quad-\mathbb{E}\int_0^tb_{ii}(s)(Y_i(s,x)Y_i(s,y)(Y_i(s,x)+Y_i(s,y)))ds\\
&\leq xy+\int_0^t\delta_i(s)\mathbb{E}(Y_i(s,x)Y_i(s,y))ds
-\int_0^tb_{ii}(s)(\mathbb{E}(Y_i(s,x)Y_i(s,y)))^{\frac{3}{2}}ds,
\end{split}
\end{equation*}
where
$\delta_i(t):=2a_i(t)+\sigma^2_i(t)+\int_{\mathbb{Y}}\gamma^2_i(t,u)\lambda(du)$.
By the comparison theorem,
\begin{equation}\label{eq07}
\begin{split}
\mathbb{E}(Y_i(t,x)Y_i(t,y))&\leq\left(1/\sqrt{xy}e^{-
\frac{1}{2}\int_0^t\delta_i(s)ds}+\frac{1}{2}\int_0^tb_{ii}(s)e^{-
\frac{1}{2}\int_s^t\delta_i(\tau)d\tau}ds\right)^{-2}\\
&\leq\left(
\hat{b}/\check{\delta}_i+(1/\sqrt{xy}-\hat{b}_{ii}/\check{\delta}_i)e^{-\frac{\check{\delta}_it}{2}}\right)^{-2}.
\end{split}
\end{equation}
 On the other hand, thanks to \eqref{eq9}
we have
\begin{equation*}
\begin{split}
&\frac{1}{Y_i(t,x)}-\frac{1}{Y_i(t,y)}\\
&\ =\left(\frac{1}{x}-\frac{1}{y}\right)
\exp\Big(-\int_0^t\Big[a_i(s)-\frac{1}{2}\sigma^2_i(s)+
\int_{\mathbb{Y}}(\ln(1+\gamma_i(s,u))-\gamma_i(s,u))\lambda(du)\Big]ds\\
&\quad-\int_0^t\sigma_i(s)dW(s)
-\int_0^t\int_{\mathbb{Y}}\ln(1+\gamma_i(s,u))\tilde{N}(ds,du)\Big)
.
\end{split}
\end{equation*}
In the same way as \eqref{eq05} was done, it follows from
\eqref{eq03} that
\begin{equation}\label{eq19}
\mathbb{E}\left|\frac{1}{Y_i(t,x)}-\frac{1}{Y_i(t,y)}\right|\leq\left|\frac{1}{x}-
\frac{1}{y}\right| e^{-c_2 t}.
\end{equation}
Thus \eqref{eq06} follows by combining \eqref{eq07} and
\eqref{eq19}.
\end{proof}

If $a_i,b_{ii},\sigma_i,\gamma_i$ are
time-independent, Eq. \eqref{eq103} reduces to
\begin{equation}\label{eq21}
dY_i(t)=Y_i(t^-)\left[(a_i-b_{ii}Y_i(t))dt+\sigma_i
dW(t)+\int_{\mathbb{Y}}\gamma_i(u)\tilde{N}(dt,du)\right],
\end{equation}
with original value $x>0$. Let $p(t,x,dy)$ denote the transition
probability of solution process $Y_i(t,x)$ and $\mathbb{P}(t,x,A)$
denote the probability of event $\{Y_i(t,x)\in A\}$, where $A$ is a
Borel measurable subset of $(0,\infty)$.  It is similar to that of Corollary \ref{exin}, under the conditions of Theorem \ref{permanent}
there exists  an invariant measure for  $Y_i(t,x)$. Moreover by the standard procedure \cite[p213-216]{my06},
we know that Theorem \ref{asymptotic} implies the uniqueness of invariant measure. That is:
\begin{theorem}\label{distribution}
{\rm Under the conditions of  Theorem \ref{permanent} and \ref{asymptotic}, the solution
$Y_i(t,x)$ of Eq. \eqref{eq21} has a unique invariant measure.}
\end{theorem}

We further need the following exponential martingale inequality with
jumps, e.g., \cite[Theorem 5.2.9, p291]{a09}.

\begin{lemma}\label{exponential martingale}
{\rm Assume that $g:[0,\infty)\rightarrow \mathbb{R}$ and
$h:[0,\infty)\times \mathbb{Y}\rightarrow \mathbb{R}$ are both
predictable $\mathcal {F}_t$-adapted processes such that for any
$T>0$
\begin{equation*}
\int_0^T|g(t)|^2dt< \infty\mbox{ a.s. and }
\int_0^T\int_{\mathbb{Y}}|h(t,u)|^2\lambda(du)dt<\infty  \mbox{
a.s.}
\end{equation*}
Then for any constants $\alpha,\beta>0$
\begin{equation*}
\begin{split}
\mathbb{P}\Big\{\sup\limits_{0\leq t\leq
T}\Big[&\int_0^tg(s)dW(s)-\frac{\alpha}{2}\int_0^t|g(s)|^2ds+\int_0^t\int_{\mathbb{Y}}h(s,u)\tilde{N}(ds,du)\\
&-\frac{1}{\alpha}\int_0^t\int_{\mathbb{Y}}[e^{\alpha
h(s,u)}-1-\alpha h(s,u)]\lambda(du)ds\Big]>\beta\Big\}\leq
e^{-\alpha\beta}.
\end{split}
\end{equation*}
}
\end{lemma}

\begin{lemma}\label{asymptotic property}
{\rm Let assumption ${\bf(A)}$ hold. Assume further that for any
$t\geq0$ and $i=1,\cdots,n$
\begin{equation}\label{eq32}
\sup\limits_{t\geq0}\int_0^t\int_{\mathbb{Y}}e^{s-t}[\gamma_i(s,u)-\ln(1+\gamma_i(s,u))]\lambda(du)ds<\infty.
\end{equation}
Then
\begin{equation*}
\limsup\limits_{t\rightarrow\infty}\frac{\ln Y_i(t)}{\ln t}\leq1,
\mbox{ a.s. for each }i=1,\cdots,n,
\end{equation*}

}
\end{lemma}

\begin{proof}
For any $t\geq0$ and $i=1,\cdots,n$, applying the It\^o formula
\begin{equation*}
\begin{split}
e^t\ln Y_i(t)&=\ln X_i(0)+\int_0^te^s\Big[\ln Y_i(s)+a_i(s)-b_{ii}(s)Y_i(s)-\frac{1}{2}\sigma_i^2(s)\\
&\quad+\int_{\mathbb{Y}}[\ln(1+\gamma_i(s,u))-\gamma_i(s,u)]\lambda(du)\Big]ds\\
&\quad+\int_0^te^s\sigma_i(s)dW(s)+\int_0^t\int_{\mathbb{Y}}e^s\ln(1+\gamma_i(s,u))\tilde{N}(ds,du).
\end{split}
\end{equation*}
Note that, for $c,x>0$, $\ln x-cx$ attains its maximum value $-1-\ln
c$ at $x=\frac{1}{c}$. Thus it follows from the inequality
\eqref{eq102} that
\begin{equation}\label{eq109}
\begin{split}
e^t\ln Y_i(t)&\leq\ln X_i(0)+\int_0^te^s\Big[-1-\ln
b_{ii}(s)+a_i(s)-\frac{1}{2}\sigma_i^2(s)
\Big]ds\\
&\quad+\int_0^te^s\sigma_i(s)dW(s)+\int_0^t\int_{\mathbb{Y}}e^s\ln(1+\gamma_i(s,u))\tilde{N}(ds,du).
\end{split}
\end{equation}
 In the light of Lemma \ref{exponential
martingale}, for any $\alpha,\beta,T>0$,
\begin{equation*}
\begin{split}
\mathbb{P}\Big\{&\sup\limits_{0\leq t\leq
T}\Big[\int_0^te^s\sigma_i(s)dW(s)-\frac{\alpha}{2}\int_0^te^{2s}
\sigma_i^2(s)ds+\int_0^t\int_{\mathbb{Y}}e^s\ln(1+\gamma_i(s,u))\tilde{N}(ds,du)\\
&-\frac{1}{\alpha}\int_0^t\int_{\mathbb{Y}}\Big[e^{\alpha
e^s\ln(1+\gamma_i(s,u))} -1-\alpha
e^s\ln(1+\gamma_i(s,u))\Big]\lambda(du)ds\Big]\geq\beta\Big\}\leq
e^{-\alpha\beta}.
\end{split}
\end{equation*}
Choose $T=k\gamma, \alpha=\epsilon e^{-k\gamma}$, and
$\beta=\frac{\theta e^{k\gamma}\ln k}{\epsilon}$, where
$k\in\mathbb{N}, 0<\epsilon<1, \gamma>0$, and $\theta>1$ in the
above equation. Since $\sum\limits_{k=1}^\infty k^{-\theta}<\infty$,
we can deduce from the Borel-Cantalli Lemma that there exists an
$\Omega_i\subseteq\Omega$ with $\mathbb{P}(\Omega_i)=1$ such that
for any $\epsilon\in\Omega_i$ an integer $k_i=k_i(\omega,\epsilon)$
can be found such that
\begin{equation*}
\begin{split}
&\int_0^te^s\sigma_i(s)dW(s)+\int_0^t\int_{\mathbb{Y}}e^s\ln(1+\gamma_i(s,u))\tilde{N}(ds,du)\\
&\quad\leq\frac{\theta e^{k\gamma}\ln k}{\epsilon}+\frac{\epsilon e^{-k\gamma}}{2}\int_0^te^{2s}\sigma_i^2(s)ds\\
&\quad\quad+\frac{1}{\epsilon
e^{-k\gamma}}\int_0^t\int_{\mathbb{Y}}\Big[(1+\gamma_i(s,u))^{\epsilon
e^{s-k\gamma}} -1-\epsilon
e^{s-k\gamma}\ln(1+\gamma_i(s,u))\Big]\lambda(du)ds
\end{split}
\end{equation*}
whenever $k\geq k_i, 0\leq t\leq k\gamma$. Next, note from the
inequality \eqref{eq100} that, for any $\omega\in\Omega_i$ and
$0<\epsilon<1,0\leq t\leq k\gamma$ with $k \geq k_i$,
\begin{equation*}
\begin{split}
&\frac{1}{\epsilon
e^{t-k\gamma}}\int_0^t\int_{\mathbb{Y}}\Big[(1+\gamma_i(s,u))^{\epsilon
e^{s-k\gamma}} -1-\epsilon e^{s-k\gamma}\ln(1+\gamma_i(s,u))\Big]\lambda(du)ds\\
&\leq\quad
\int_0^t\int_{\mathbb{Y}}e^{s-t}(\gamma_i(s,u)-\ln(1+\gamma_i(s,u)))\lambda(du)ds.
\end{split}
\end{equation*}
Thus, for $\omega\in\Omega_i$ and $(k-1)\gamma\leq t\leq k\gamma$
with $k\geq k_i+1$, we have
\begin{equation*}
\begin{split}
\frac{\ln Y_i(t)}{\ln t}&\leq \frac{\ln X_i(0)}{e^t\ln
t}+\frac{\theta e^{k\gamma}\ln k}{\epsilon
e^{(k-1)\gamma}\ln((k-1)\gamma)}\\
&\quad+\frac{1}{\ln t}\int_0^te^{s-t}\Big[-1-\ln
b_{ii}(s)+a_i(s)-\frac{1}{2}(1-\epsilon e^{s-k\gamma})\sigma_i^2(s)
\Big]ds\\
&\quad+\frac{1}{\ln
t}\int_0^t\int_{\mathbb{Y}}e^{s-t}[\gamma_i(s,u)-\ln(1+\gamma_i(s,u))]\lambda(du)ds.
\end{split}
\end{equation*}
Letting $k\uparrow\infty$,together with assumption ${\bf(A)}$ and
\eqref{eq32}, leads to
\begin{equation*}
\limsup\limits_{t\rightarrow\infty}\frac{\ln Y_i(t)}{\ln
t}\leq\frac{\theta e^{\gamma}}{\epsilon},
\end{equation*}
and the conclusion follows by setting
$\gamma\downarrow0,\epsilon\uparrow1$, and $\theta\downarrow1$.
\end{proof}

Noting the limit $\lim\limits_{t\rightarrow\infty}\dfrac{\ln
t}{t}=0$, we have the following corollary.
\begin{corollary}\label{corollary}
{\rm Under the conditions of Lemma \ref{asymptotic property}
\begin{equation*}
\limsup\limits_{t\rightarrow\infty}\dfrac{\ln Y_i(t)}{t}\leq0,
\mbox{ a.s. for each }i=1,\cdots,n,
\end{equation*}
and therefore
\begin{equation*}
\limsup\limits_{t\rightarrow\infty}\dfrac{\ln\Big(\prod\limits_{i=1}^nY_i(t)\Big)}{t}\leq0,
\mbox{ a.s.}
\end{equation*}
}

\end{corollary}

\begin{corollary}
{\rm Under the conditions of Lemma \ref{asymptotic property}
\begin{equation*}
\limsup\limits_{t\rightarrow\infty}\dfrac{\ln(X_i(t))}{t}\leq0,
\mbox{ a.s. for each }i=1,\cdots,n,
\end{equation*}
and therefore
\begin{equation*}
\limsup\limits_{t\rightarrow\infty}\dfrac{\ln\Big(\prod\limits_{i=1}^nX_i(t)\Big)}{t}\leq0,
\mbox{ a.s.}
\end{equation*}
}
\end{corollary}

\begin{proof}
Recalling
\begin{equation*}
Z_i(t)\leq X_i(t)\leq Y_i(t),t\geq0,i=1,\cdots,n
\end{equation*}
and combining Corollary \ref{corollary}, we complete the proof.

\end{proof}

\begin{theorem}\label{YT}
{\rm Let the conditions of Lemma \ref{asymptotic property} hold.
Assume further that for any $t\geq0$ and $i=1,\cdots,n$
\begin{equation}\label{eq105}
R_i(t):=a_i(t)-\frac{1}{2}\sigma^2(t)+
\int_{\mathbb{Y}}(\ln(1+\gamma_i(t,u))-\gamma_i(t,u))\lambda(du)\geq0,
\end{equation}
and there exists constant $c_2>0$ such that
\begin{equation}\label{eq113}
\int_{\mathbb{Y}}(\ln(1+\gamma_i(t,u)))^2\lambda(du)\leq c_2.
\end{equation}
 Then for each  $i=1,\cdots,n$
\begin{equation}
\lim\limits_{t\rightarrow\infty}\frac{\ln Y_i(t)}{t}=0 \mbox{ a.s. }
\end{equation}
}
\end{theorem}

\begin{proof}
According to Corollary \ref{corollary}, it suffices to show
$\liminf\limits_{t\rightarrow\infty}\frac{\ln Y_i(t)}{t}\geq0$.
Denote for $t\geq0$
\begin{equation*}
M_i(t):=\int_0^t\sigma_i(s)dW(s) \mbox{ and }
\bar{M}_i(t):=\int_0^t\int_{\mathbb{Y}}\ln(1+\gamma_i(s,u))\tilde{N}(ds,du).
\end{equation*}
Note that
\begin{equation*}
[M_i](t)=\langle
M_i\rangle(t)=\int_0^t\sigma_i^2(s)ds\leq\check{\sigma}_i^2t,
\end{equation*}
and by \eqref{eq113}
\begin{equation*}
\langle
\bar{M}_i\rangle(t)=\int_0^t\int_{\mathbb{Y}}(\ln(1+\gamma_i(s,u)))^2\lambda(du)ds\leq
c_2 t.
\end{equation*}
Since
\begin{equation*}
\int_0^t\frac{1}{(1+s)^2}ds=-\frac{1}{1+s}\Big|_{0}^t=\frac{t}{1+t}<\infty,
\end{equation*}
together with  Lemma \ref{large numbers}, we then obtain
\begin{equation}\label{eq121}
\lim\limits_{t\rightarrow\infty}\frac{1}{t}\int_0^t\sigma_i(s)dW(s)=0
\mbox{ a.s. and
}\lim\limits_{t\rightarrow\infty}\frac{1}{t}\int_0^t\int_{\mathbb{Y}}\ln(1+\gamma_i(s,u))\tilde{N}(ds,du)=0
\mbox{ a.s. }
\end{equation}
Moreover, it is easy to see that for any $t>s$
\begin{equation*}
\int_s^t\sigma_i(r)dW(r)=\int_0^t\sigma_i(r)dW(r)-\int_0^s\sigma_i(r)dW(r)
\end{equation*}
and
\begin{equation*}
\int_s^t\int_{\mathbb{Y}}\ln(1+\gamma_i(r,u))\tilde{N}(dr,du)=\int_0^t\int_{\mathbb{Y}}\ln(1+\gamma_i(r,u))\tilde{N}(dr,du)-
\int_0^s\int_{\mathbb{Y}}\ln(1+\gamma_i(r,u))\tilde{N}(dr,du).
\end{equation*}
Consequently, for any $\epsilon>0$  we can deduce that there exists
constant $T>0$ such that
\begin{equation}\label{eq106}
\left|\int_s^t\sigma_i(r)dW(r)\right|\leq\epsilon(s+t) \mbox{ a.s.
}\mbox{ and
}\left|\int_s^t\int_{\mathbb{Y}}\ln(1+\gamma_i(r,u))\tilde{N}(dr,du)\right|\leq\epsilon(s+t)\mbox{
a.s. }
\end{equation}
whenever $t>s\geq T$. Furthermore, by Lemma \ref{explicit solution},
together with \eqref{eq106}, we have for $t\geq T$
\begin{equation*}
\begin{split}
\frac{1}{Y_i(t)}
&\leq\frac{1}{Y_i(T)}\exp\Big(\int_T^t-\Big[a_i(s)-\frac{1}{2}\sigma^2_i(s)+
\int_{\mathbb{Y}}(\ln(1+\gamma_i(s,u))-\gamma_i(s,u))\lambda(du)\Big]ds\\
&\quad+2\epsilon(t+T)\Big)\\
&\quad+\int_T^tb_{ii}(s)\exp\Big(-\int_s^t\Big[a_i(r)-\frac{1}{2}\sigma^2_i(r)+
\int_{\mathbb{Y}}(\ln(1+\gamma_i(r,u))-\gamma_i(r,u))\lambda(du)\Big]dr\\
&\quad+2\epsilon(s+t)\Big)ds, \mbox{ a.s. }
\end{split}
\end{equation*}
This further gives that for any $t\geq T$
\begin{equation*}
\begin{split}
e^{-4\epsilon(t+T)}\frac{1}{Y_i(t)}&\leq\frac{1}{Y_i(T)}\exp\Big(\int_T^t-\Big[a_i(s)-\frac{1}{2}\sigma^2(s)+
\int_{\mathbb{Y}}(\ln(1+\gamma_i(s,u))-\gamma_i(s,u))\lambda(du)\Big]ds\\
&+\int_T^tb_{ii}(s)\exp\Big(-\int_s^t\Big[a_i(r)-\frac{1}{2}\sigma^2_i(r)+
\int_{\mathbb{Y}}(\ln(1+\gamma_i(r,u))-\gamma_i(r,u))\lambda(du)\Big]dr\\
&\quad-2\epsilon(t-s)-2\epsilon T\Big)ds, \mbox{ a.s. }
\end{split}
\end{equation*}
Thus in view of \eqref{eq105} there exists constant $K>0$ such that
for any $t\geq T$
\begin{equation*}
e^{-4\epsilon(t+T)}\frac{1}{Y_i(t)}\leq K, \mbox{ a.s. }
\end{equation*}
Hence for any $t\geq T$
\begin{equation*}
\frac{1}{t}\ln\frac{1}{Y_i(t)}\leq
4\epsilon\Big(1+\frac{T}{t}\Big)+\frac{1}{t}\ln K, \mbox{ a.s. }
\end{equation*}
and the conclusion follows by letting $t\rightarrow\infty$ and the
arbitrariness of $\epsilon>0$.
\end{proof}

\subsection{Further Properties of $n-$Dimensional Competitive Models}

We need the following lemma.

\begin{lemma}\label{lemma1}
{\rm Let the conditions of Theorem \ref{YT} hold. Assume further
that for $i,j=1,\cdots,n$
\begin{equation}\label{eq120}
R_{ij}:=\sup\left\{\frac{b_{ij}(t)}{b_{jj}(t)}, t\geq0,i\neq
j\right\}
\end{equation}
satisfy
\begin{equation}\label{eq116}
R_i(t)-\sum\limits_{i\neq j}R_{ij}R_j(t)>0,  \, t\geq0.
\end{equation}
Then
\begin{equation}
\liminf\limits_{t\rightarrow\infty}\frac{\ln Z_i(t)}{t}\geq0, \mbox{
a.s. }
\end{equation}
where $Z_i(t),  i=1,\cdots,n$ are solutions of \eqref{eq112}.}
\end{lemma}

\begin{remark}
{\rm For $i,j=1,\cdots,n$ and $t\geq0$, if $ b_{ij}(t)$ takes
finite-number values, then condition \eqref{eq120} must hold.

}

\end{remark}

\begin{proof}
It is sufficient to show
$\limsup\limits_{t\rightarrow\infty}\frac{1}{t}\ln
\frac{1}{Z_i(t)}\leq0$. Note from Lemma \ref{explicit solution} that
for any $t>s\geq0$
\begin{equation}\label{eq115}
\begin{split}
\frac{1}{Z_i(t)}&=\frac{1}{Z_i(s)}\exp\Big(\int_s^t-\Big[a_i(r)-\sum\limits_{i\neq
j}b_{ij}(r)Y_j(r)-\frac{1}{2}\sigma^2_i(r)+
\int_{\mathbb{Y}}(\ln(1+\gamma_i(r,u))-\gamma_i(r,u))\lambda(du)\Big]dr\\
&\quad-\int_s^t\sigma_i(s)dW(s)
-\int_s^t\int_{\mathbb{Y}}\ln(1+\gamma_i(s,u))\tilde{N}(ds,du)\Big)\\
&\quad+\int_s^tb_{ii}(r)\exp\Big(-\int_r^t\Big[a_i(\tau)-\sum\limits_{i\neq
j}b_{ij}(\tau)Y_j(\tau)-\frac{1}{2}\sigma^2_i(\tau)\\
&\quad+
\int_{\mathbb{Y}}(\ln(1+\gamma_i(\tau,u))-\gamma_i(\tau,u))\lambda(du)\Big]d\tau\\
&\quad-\int_r^t\sigma_i(\tau)dW(\tau)
-\int_r^t\int_{\mathbb{Y}}\ln(1+\gamma_i(\tau,u))\tilde{N}(d\tau,du)\Big)dr.
\end{split}
\end{equation}
Applying the It\^o formula, for any $t>s\geq0$
\begin{equation}\label{eq107}
\begin{split}
\int_s^tb_{ii}(r)Y_i(r)dr&=\ln Y_i(s)-\ln
Y_i(t)\\
&\quad+\int_s^t\left[a_i(r)-\frac{1}{2}\sigma_i^2(r)+\int_{\mathbb{Y}}(\ln(1+\gamma_i(r,u))-\gamma_i(r,u))\lambda(du)\right]ds\\
&\quad+\int_s^t\sigma_i(r)dW(r)+\int_s^t\int_{\mathbb{Y}}\ln(1+\gamma_i(r,u))\tilde{N}(dr,du).
\end{split}
\end{equation}
This, together with Theorem \ref{YT} and \eqref{eq106}, yields that
for any $\epsilon>0$ there exists $\bar{T}>0$ such that
\begin{equation}\label{eq114}
\begin{split}
\int_s^tb_{ii}(r)Y_i(r)dr&\leq\int_s^t\left[a_i(r)-\frac{1}{2}\sigma_i^2(r)+\int_{\mathbb{Y}}(\ln(1+\gamma_i(r,u))-\gamma_i(r,u))\lambda(du)\right]ds\\
&\quad+3\epsilon(s+t)\\
\end{split}
\end{equation}
whenever $t\geq s\geq\bar{T}$. Moreover taking into account
\eqref{eq107} and \eqref{eq114}, we have for $t>s\geq\bar{T}$
\begin{equation*}
\begin{split}
\int_s^tb_{ij}(r)Y_j(r)dr&=\int_s^t\frac{b_{ij}(r)}{b_{jj}(r)}b_{jj}(r)Y_j(r)dr\\
&\leq R_{ij}\int_s^tb_{jj}(r)Y_j(r)dr\\
&\leq 3\epsilon(s+t)R_{ij}\\
&\quad+\int_s^tR_{ij}\left[a_i(r)-\frac{1}{2}\sigma_i^2(r)+\int_{\mathbb{Y}}(\ln(1+\gamma_i(r,u))-\gamma_i(r,u))\lambda(du)\right]ds.
\end{split}
\end{equation*}
Putting this into \eqref{eq115} leads to
\begin{equation*}
\begin{split}
\frac{1}{Z_i(t)}&=\frac{1}{Z_i(s)}\exp\Big(-\int_s^t\Big[R_i(r)-\sum\limits_{i\neq j}R_{ij}R_j(r)\Big]dr+\epsilon(s+t)\Big(3\sum\limits_{i\neq j}R_{ij}+2\Big)\Big)\\
&\quad+\int_s^tb_{ii}(r)\exp\Big(-\int_r^t\Big[R_i(\tau)-\sum\limits_{i\neq j}R_{ij}R_j(\tau)\Big]d\tau\\
&\quad+\epsilon(r+t)\Big(3\sum\limits_{i\neq j}R_{ij}+2\Big)\Big)dr,
\end{split}
\end{equation*}
which, in addition to \eqref{eq116}, implies
\begin{equation*}
\begin{split}
\frac{1}{Z_i(t)}&=\frac{1}{Z_i(s)}\exp\Big(\epsilon(s+t)\Big(3\sum\limits_{i\neq
j}R_{ij}+2\Big)\Big)+\int_s^tb_{ii}(r)\exp\Big(\epsilon(r+t)\Big(3\sum\limits_{i\neq
j}R_{ij}+2\Big)\Big)dr.
\end{split}
\end{equation*}
Carrying out similar arguments to Theorem \ref{YT}, we can deduce
that there exists $K>0$ such that for $t>s\geq \bar{T}$
\begin{equation*}
\exp\Big(-2\epsilon(s+t)\Big(3\sum\limits_{i\neq
j}R_{ij}+2\Big)\Big)\frac{1}{Z_i(t)}\leq K
\end{equation*}
and the conclusion follows.
\end{proof}

Now a combination of Theorem \ref{YT} and Lemma \ref{lemma1} gives
the following theorem.

\begin{theorem}
{\rm Under the conditions of Lemma \ref{lemma1}, for each
$i=1,\cdots,n$
\begin{equation*}
\lim\limits_{t\rightarrow\infty}\frac{\ln X_i(t)}{t}=0, \mbox{ a.s.
}
\end{equation*}
}
\end{theorem}

Another important property of a population dynamics is the
extinction which means every species will become extinct. The most
natural analogue for the stochastic population dynamics \eqref{eq42}
is that every species will become extinct with probability $1$. To
be precise, let us give the definition.

\begin{definition}
{\rm Stochastic population dynamics \eqref{eq42} is said to be
extinct with probability $1$ if, for every initial data
$x_0\in\mathbb{R}^n_+$, the solution $X_i(t),t\geq0$, has the
property
\begin{equation*}
\lim\limits_{t\rightarrow\infty}X_i(t)\rightarrow0 \ \ \ \mbox{ a.s.
}.
\end{equation*}
}
\end{definition}

\begin{theorem}
{\rm Let assumption ${\bf(A)}$ and \eqref{eq113} hold. Assume
further that
\begin{equation*}
\eta_i:=\limsup\limits_{t\rightarrow\infty}\frac{1}{t}\int_0^t\beta_i(s)ds<0,
\end{equation*}
where, for $t\geq0$ and $i=1,\cdots,n$,
\begin{equation*}
\beta_i(t):=a_i(t)-\frac{1}{2}\sigma^2_i(t)-
\int_{\mathbb{Y}}(\gamma_i(t,u)-\ln(1+\gamma_i(t,u)))\lambda(du).
\end{equation*}
Then stochastic population dynamics \eqref{eq42} is extinct a.s.}
\end{theorem}
\begin{proof}
Recalling by the comparison theorem that, for any $t\geq0$ and
$i=1,\cdots,n$,
\begin{equation*}
X_i(t)\leq Y_i(t),
\end{equation*}
we only need to verify $\lim\sup_{t\rightarrow\infty}Y_i(t)=0$
\mbox{ a.s.}, due to
\begin{equation*}
0\leq\liminf_{t\rightarrow\infty}X_i(t)\leq\limsup_{t\rightarrow\infty}X_i(t)\leq\limsup_{t\rightarrow\infty}Y_i(t).
\end{equation*}
Since $b_i(t)\geq0$, by \eqref{eq9} it is easy to deserve that
\begin{equation*}
\begin{split}
Y_i(t)&\leq
X_i(0)\exp\Big(\int_0^t\beta_i(s)ds+\int_0^t\sigma_i(s)dW(s)
+\int_0^t\int_{\mathbb{Y}}\ln(1+\gamma_i(s,u))\tilde{N}(ds,du)\Big)\\
&=X_i(0)\exp\Big(t\Big(\frac{1}{t}\int_0^t\beta_i(s)ds+\frac{1}{t}\int_0^t\sigma_i(s)dW(s)\\
&\quad+\frac{1}{t}\int_0^t\int_{\mathbb{Y}}\ln(1+\gamma_i(s,u))\tilde{N}(ds,du)\Big)\Big).
\end{split}
\end{equation*}
Thanks to $\eta_i<0$, in addition to \eqref{eq121}, we deduce that
$\lim\sup_{t\rightarrow\infty}Y_i(t)=0$ \mbox{ a.s. } and the
conclusion follows.

\end{proof}

\begin{remark}
{\rm In Theorem \ref{distribution}, we know that one dimensional our model has a unique invariant measure under some conditions, however we can not obtain the same result for $n$dimensional model ($n\ge 2$)}.
\end{remark}

\section{Conclusions and Further Remarks}

In this paper, we discuss competitive Lotka-Volterra population
dynamics with jumps. We show that the model admits a unique global
positive solution, investigate uniformly finite $p$-th moment with
$p>0$, stochastic ultimate boundedness, invariant measure and
long-term behaviors of solutions. Moreover, using a
variation-of-constants formula for a class of SDEs with jumps, we
provide explicit solution for the model, investigate precisely the
sample Lyapunov exponent for each component and  the extinction of
our $n$-dimensional model.

As we mentioned in the introduction section, random perturbations of
interspecific or intraspecific interactions by white noise is one of
ways to perturb population dynamics. In \cite{mmr02}, Mao, et al.
investigate
 stochastic $n$-dimensional Lotka-Volterra systems
\begin{equation}\label{eq53}
dX(t)=\mbox{diag}(X_1(t),\cdots,X_n(t))\left[(a+BX(t))dt+\sigma
X(t)dW(t)\right],
\end{equation}
where $a=(a_1,\cdots,a_n)^T, B=(b_{ij})_{n\times n},
\sigma=(\sigma_{ij})_{n\times n}$. It is interesting to know what
would happen if stochastic Lotka-Volterra systems \eqref{eq53} are
further perturbed by jump diffusions, namely
\begin{equation}\label{eq54}
\begin{split}
d X(t)&=\mbox{diag}(X_1(t^-),\cdots,X_n(t^-))\Big[(a+BX(t))dt+\sigma
X(t)d W(t)\\
&\quad+\int_{\mathbb{Y}}\gamma(X(t^-),u)\tilde{N}(dt,du)\Big],
\end{split}
\end{equation}
where $\gamma=(\gamma_1,\cdots,\gamma_n)^T$. On the other hand, the
hybrid systems driven by continuous-time Markov chains have been
used to model many practical systems where they may experience
abrupt changes in their structure and parameters caused by phenomena
such as environmental disturbances \cite{my06}. As mentioned in Zhu
and Yin \cite{zy09a, zy09b}, interspecific or intraspecific
interactions are often subject to environmental noise, and the
qualitative changes cannot be described by the traditional
(deterministic or stochastic) Lotka-Volterra models. For example,
interspecific or intraspecific interactions often vary according to
the changes in nutrition and/or food resources. We use the
continuous-time Markov chain $r(t)$ with a finite state space
$\mathcal {M}=\{1,\cdots, m\}$ to model these abrupt changes, and
need to deal with stochastic hybrid population dynamics with jumps
\begin{equation}\label{eq55}
\begin{split}
d
X(t)&=\mbox{diag}(X_1(t^-),\cdots,X_n(t^-))\Big[(a(r(t))+B(r(t))X(t))dt+\sigma(r(t))
X(t)d W(t)\\
&\quad+\int_{\mathbb{Y}}\gamma(X(t^-),r(t),u)\tilde{N}(dt,du)\Big].
\end{split}
\end{equation}
We will report these in our following papers.

\end{document}